\documentclass[12pt,reqno]{amsart}

\usepackage{amsfonts,amssymb}
\usepackage{times,amsmath,amstext,amsbsy,amsopn,amsthm,amscd,eucal}
\usepackage[all]{xy}
\usepackage{enumerate}

\allowdisplaybreaks

\newtheorem{theorem}{Theorem}[section]
\newtheorem{Luna Slice Theorem}[theorem]{Luna Slice Theorem}

\newtheorem{corollary}[theorem]{Corollary}
\newtheorem{lemma}[theorem]{Lemma}
\newtheorem{proposition}[theorem]{Proposition}
\newtheorem{proposition-definition}[theorem]{Proposition-Definition}

\theoremstyle{definition}
\newtheorem{definition}[theorem]{Definition}

\theoremstyle{remark}
\newtheorem{remark}[theorem]{Remark}

\newcommand{\fatdot}{{\scriptscriptstyle \bullet}}
\newcommand{\Ob}{\operatorname{Ob}\nolimits}

\newcommand{\rk}{\operatorname{rk}\nolimits}
\newcommand{\Hom} {\operatorname{Hom}\nolimits}

\newcommand{\Aut} {\operatorname{Aut}\nolimits}
\newcommand{\GL} {\operatorname{GL}\nolimits}
\newcommand{\End} {\operatorname{End}\nolimits}
\newcommand{\Stab} {\operatorname{Stab}\nolimits}
\newcommand{\PGL} {\operatorname{PGL}\nolimits}

\newcommand\CC{{\mathbb C}}

\newcommand\QQ{{\mathbb Q}}

\newcommand\PP{{\mathbb P}}
\newcommand\ZZ{{\mathbb Z}}

\newcommand\HH{{\mathbb H}}
\newcommand\NN{{\mathbb N}}
\newcommand{\YYY}{{\mathcal Y}}

\newcommand{\GGG}{{\mathcal G}}
\newcommand{\UUU}{{\mathcal U}}
\newcommand{\LLL}{{\mathcal L}}
\newcommand{\DDD}{{\mathcal D}}
\newcommand{\FFF}{{\mathcal F}}
\newcommand{\KKK}{{\mathcal K}}

\newcommand{\MMM}{{\mathcal M}}
\newcommand{\CCC}{{\mathcal C}}
\newcommand{\III}{{\mathcal I}}
\newcommand{\OOO}{{\mathcal O}}
\newcommand{\EEE}{{\mathcal E}}
\newcommand{\WWW}{{\mathcal W}}

\newcommand{\TTT}{{\mathcal T}}

\newcommand{\HHH}{{\mathcal H}}
\renewcommand\phi{\varphi}

\newcommand{\gr}{\operatorname{gr}\nolimits}

\newcommand{\At}{\operatorname{At}\nolimits}
\newcommand{\Sch}{\operatorname{Sch}\nolimits}

\newcommand{\ud}{\mathrm{d}}

\newcommand{\Mor}{\operatorname{Mor}\nolimits}
\newcommand{\Spec}{\operatorname{Spec}\nolimits}
\newcommand{\ob}{\operatorname{ob}\nolimits}

\newcommand{\im}{\operatorname{im}\nolimits}
\newcommand{\id}{\operatorname{id}\nolimits}
\newcommand{\Pic}{\operatorname{Pic}\nolimits}

\newcommand{\Tr}{\operatorname{Tr}\nolimits}

\newcommand{\Quot}{\operatorname{Quot}\nolimits}

\newcommand{\Ext}{\operatorname{Ext}\nolimits}
\newcommand{\Supp}{\operatorname{Supp}\nolimits}

\newcommand\lra{{\longrightarrow}}

\newcommand\ra{{\rightarrow}}
\newcommand\rar{{\rightarrow}}

\renewcommand\emptyset{\varnothing}

\begin{document}
\title{Local structure of Moduli Spaces}
\author{Francois-Xavier Machu}
\address{F-X. M.: Department of Mathematical and Statistical Sciences,
632 CAB, University of Alberta, Edmonton, AB T6G 2G1, Canada}
\email{machu@ualberta.ca}

\bigskip
\bigskip
\begin{abstract}
We provide a sketch of the GIT construction of the moduli spaces for the three classes of connections: the class of meromorphic connections 
with fixed divisor of poles $D$ and its subclasses of integrable and integrable logarithmic connections. We use the Luna Slice Theorem
to represent the germ of the moduli space as the quotient of the Kuranishi space by the automorphism group of the central fiber.
This method is used to determine the singularities of the moduli space of connections in some examples.
\end{abstract}
\maketitle
\vspace{1 ex}
\begin{center}\begin{minipage}{110mm}\footnotesize{\bf Key words:} boundedness, connections, differential operators, good quotients, Luna slices, moduli spaces, singularities, stability, versal deformations.
\end{minipage}
\end{center}

\begin{center}\begin{minipage}{110mm}\footnotesize{\bf MSC2000:} 14B12, 14D22, 14F20, 14H60, 32G34. 

\end{minipage}
\end{center}
\vspace{1 ex}

\section*{Introduction}
The main objective of the present paper is an application of the GIT method and Kuranishi spaces to the study
of the local structure of moduli spaces of connections. This is a generalization of the approach
that has been used before by several authors for the study of other moduli spaces, namely, moduli spaces
of semistable sheaves \cite{O'G}, \cite{Dr}, \cite{L-S}, \cite{M-T}. The general idea on the moduli spaces of connections
$(\EEE,\nabla)$ as affine fibers bundles over the moduli spaces of vector bundles $\EEE$ fails at the boundary,
when $\EEE$ is nomore (semi)stable whilst that the pair $(\EEE,\nabla)$ is still semi(stable).
Our approach is conceived for the study of such phenemona occuring at the boundary of the moduli spaces,
and we produce some examples of them.

We consider the following classes of connections over a polarized projective variety $X$:
all the connections with fixed divisor of poles $D$, integrable ones,
integrable logarithmic connections. In addition to the above assumptions, we assume in the logarithmic case, that $X$ is smooth and $D$ is a simple normal crossing divisor.
In each of these classes, there exists an appropriate notion of stability
of $(\EEE,\nabla)$, and the moduli space of stable objects can be constructed
as a GIT quotient under an action of $GL(k)$ for some (big) $k$.
This can be seen for the integrable case
by an easy modification of the proof of Simpson \cite{Sim}, originally written for
regular integrable connections (that is, $D=0$). For the logarithmic case, the moduli
space was constructed by Nitsure \cite{Ni}. The moduli space
for logarithmic connections with parabolic structure at poles was constructed in
\cite{I-Iw-S} for the case $\dim X=1$. It turns out that the nonintegrable case is completely similar to
the integrable case and use this occasion to state the relevant results on the existence of quasi-projective schemes at no additional cost.

In all these cases a standard argument using the Luna slice theorem provides
a versal deformation of $(\EEE,\nabla)$ whose base $\WWW$ is an affine scheme
endowed with an action of the group $H=\Aut (\EEE,\nabla)=\Stab_{GL(k)}(\EEE,\nabla)$,
and the germ of the moduli space at $[(\EEE,\nabla)]$ is isomorphic to the
germ of the GIT quotient $\WWW// H$. On the other hand, the Kuranishi space $\KKK$
of $(\EEE,\nabla)$ is the formal completion of $\WWW$; it carries a natural
action of $H$, and the quotient $\KKK//H$ is the formal neighborhood of
$[(\EEE,\nabla)]$ in the moduli space. We use this method to determine the singularities of the moduli space of connections
in some examples.

The problem of studying moduli spaces of connections is interesting in both integrable and nonintegrable cases.
In the integrable case, there are a lot of works on the moduli 
spaces of meromorphic connections (with either regular or irregular singularities), and the moduli problem is often regarded as
solved when a parameterization is given by some multi-valued analytic functions \cite{Kor-1}, \cite{Kor-2}, \cite{EG} and \cite{Gru-Kri}. 
It is interesting to study the same moduli problem  in the algebraic framework, where the importance of moduli spaces is due to the fact
that they appear as objects of the geometric Langlands program.
The most intriguing question is the relation between the moduli space of connections and that
of underlying vector bundles. The natural map forgetting the second component of the pair 
$(\EEE,\nabla)$ is only defined over the locus of semistable vector bundles for these ones
have a consistent moduli theory. Our method provides the techniques to describe the local structure of this map
over the boundary of the moduli space of stable vector bundles.\\
The interest to nonintegrable connections in the framework of  algebraic geometry was first aroused by the work of
Bloch-Esnault \cite{Blo-Es}.  This field is rich in problems and unexplored phenomena. As note the authors
of \cite{Blo-Es}, no a single example of regular nonintegrable connection over a smooth projective variety is known.
There are a lot of meromorphic ones and an interesting problem is to study the higher Abel Jacobi maps
from our moduli spaces towards the Deligne cohomology given by differential characters of Bloch-Esnault.
Another occurence of nonintegrable connections in algebraic geometry is related to Balaji-Kollar \cite{Ba-Kol},
Balaji-Parameswaran \cite{Ba-Par}, where the algebraic holonomy group of a variety is constructed, and it is interesting
to know whether the latter has a description in terms of nonintegrable connections.

In Sect. 1, we recall the GIT construction of moduli spaces of sheaves. To construct moduli spaces of connections, we follow
the approach of \cite{Sim}. He introduces the notion of a sheaf of rings
of differential operators $\Lambda$ over a projective scheme $X$ and constructs quasi-projective moduli schemes
of (semi)-stable coherent $\Lambda$-modules. Plugging in different $\Lambda$'s, one obtains moduli spaces
of sheaves ($\Lambda=\OOO_X$), integrable regular connections ($\Lambda=\DDD_X$, the standard sheaf of differential operators),
integrable logarithmic connections ($\Lambda\subset\DDD_X$ is the sheaf of subrings generated by logarithmic vector field
$\TTT_X\langle D\rangle$, where $X$ is assumed to be smooth and $D$ is a normal crossing divisor), Higgs bundles ($\Lambda=\gr\fatdot\DDD_X=
\oplus_{m={0}}^{\infty}S^m\TTT_X)$ and others.

We extend the approach of Simpson to the case of non-integrable connections, regular or meromorphic with fixed divisor of poles $D$.
To this end, we had to slightly generalize Simpson's notion of a sheaf of rings of differential operators.
As in Simpson's definition, our $\Lambda$ is a filtered $\OOO_X$-bialgebra, satisfying a bunch of axioms (see Sect. \ref{Moduli of connections}), but contrary to Simpson, we do not assume that the graded ring $\gr\fatdot\Lambda$ is commutative. Thus to obtain moduli of 
regular connections on $X$, we set $\Lambda$ to be the sheaf $\boldsymbol{\DDD}_X$ of noncommutative differential operators
(the basic vector fields $\frac{\partial}{\partial x_i}$ associated to some coordinates $(x_1,\dots,x_n)$ do not
satisfy the commutativity relation $\frac{\partial}{\partial x_i}\frac{\partial}{\partial x_j}=\frac{\partial}{\partial x_j}\frac{\partial}{\partial x_i}$). To deal with the meromorphic connections with fixed divisor of poles $D$, we choose for $\Lambda$ the sheaf of subalgebras
of $\boldsymbol{\DDD}_X$ generated by the subsheaf $\TTT_X(-D)$ of $\TTT_X$, consisting of vector field vanishing on $D$.

In order to state the results on moduli spaces, we first introduce the relevant terminology: stable, semistable objects, representable
and corepresentable functors, coarse and fine moduli spaces, Mumford's $m$-regularity and Grothendieck's Quot-scheme.
We base essentially upon the monograph \cite{H-L}, but also use \cite{Sim}, \cite{Ma-1} and \cite{Ma-2}.

In Sect. 2, we introduce the sheaves of rings of differential operators $\Lambda$ to the Simpson and explain how
the ingredients of Simpson's proof are adapted to the nonintegrable case.

In Sect. 3, we state the Luna slice theorem and show that, in our definition of the moduli space of $\Lambda$-modules, it provides a versal deformation of a $\Lambda$-module.

In Sect. 4, we apply the Luna slice theorem to compute, in some examples, the germ of a moduli space of connections. We put in evidence the situations when the map from the moduli space
of connections to that of underlying vector bundles is either undefined, or is not a locally trivial fibration.
An example is given when that the moduli space of semistable connections is a partial compactification
of the locally trivial affine bundle of connections over the stable locus $M_X^s(r,d)$ of the moduli space of vector bundles
"in the limit $\nabla\rar\infty."$

\section{Moduli of sheaves}
\bigskip
In this section, we introduce basic notions on the moduli space of sheaves. We will
follow \cite{H-L}, \cite{Ma-1}, \cite{Ma-2}, \cite{Sim}.
Throughout the section, $X$ will be a polarized projective scheme over a field $k$,
$\OOO_X(1)$ some fixed ample invertible sheaf on $X$.
We will identify vector bundles with their associated locally free sheaves.

\subsection{Stability}
For  any coherent sheaf $E$ on $X$, one defines:
\begin{definition}
The support of $E$ is the closed set $\Supp(E)=\{x\in X\mid E_x\neq 0\}$.
Its dimension is called the dimension of the sheaf $E$ and is denoted by $\dim(E)$.
\end{definition}
\begin{definition}
$E$ is pure of dimension $d$ if $\dim(F)=d$ for all non-trivial coherent 
subsheaves $F\subset E$.
\end{definition}
We only consider the equidimensional coherent sheaves on a polarized projective scheme (or variety) over $k$.

Recall that the Euler characteristic of 
a pure coherent sheaf $E$ of $\OOO_X$-modules on $X$ is $$\chi(E):=\sum_{i=0}^{\dim(E)} (-1)^{i}h^{i}(X, E),$$ where $h^{i}(X, E)=\dim H^{i}(X, E).$
If we fix an ample line bundle $\OOO_X (1)$ on X, then the Hilbert polynomial $P_E$ is given by $$m\mapsto\chi(E\otimes\OOO_X(m));$$
In particular, $P(E)$ can be uniquely written in the form $$P(E, m)=\sum_{i=0}^{\dim(E)} \alpha_i (E)\frac{m^{i}}{i!},$$ with
integer coefficients $\alpha_i (E)$. Futhermore, if $E\neq 0$, the leading coefficient $\alpha_{\dim(E)}$, called the multiplicity, is
always positive.
\begin{definition}
If $E$ is a pure coherent sheaf of $\OOO_X$-modules on $X$ of dimension $d=\dim(X)$, then 
$$\rk(E):=\frac{\alpha_d (E)}{d!}$$ is called the rank of $E$ and $$\deg(E):=\alpha_{d-1}(E)-\rk(E)(d-1)!$$
is called the degree of $E$. The slope of $E$ is $$\mu(E):=\frac{\deg (E)}{\rk(E)}.$$
\begin{remark}
On a smooth projective variety, the Hirzebruch-Riemann-Roch formula shows that $\deg(E)=c_1(E).H^{d-1}$, where $H$ is an ample divisor
defined by a section of $\OOO_{X}(1)$.
\end{remark}
\end{definition}
\begin{definition}
The reduced Hilbert polynomial $p(E)$ ($E\neq0$) is defined by $$p(E, m):=\frac{P(E, m)}{\alpha_{d}(E)}.$$
\end{definition}
\begin{definition}
A pure coherent sheaf $E$ of $\OOO_X$-modules on a scheme $X$ of  dimension $d=\dim(X)$ is semistable (resp. stable) 
if and only if for any proper subsheaf $F$ of $E$, one has $p(F)\leq p(E)$ (resp. $p(F)<p(E)$).
\end{definition} 
\begin{remark}
We here have used the relation of lexicographic order of their coefficients on the polynomial ring $\QQ[m]$.
Explicitly, $f\leq g$ if and only if $f(m)\leq g(m)$ for $m>>0$. Analogously, $f<g$ if and only if $f(m)<g(m)$ for $m>>0$. 
\end{remark}
\begin{definition}
A pure coherent sheaf $E$ of $\OOO_X$-modules on $X$ of dimension $d=\dim(X)$ is $\mu$-
stable (resp. $\mu$-semistable) if $\mu(F)<\mu(E)$ (resp. $\mu(F)\leq\mu(E)$) for all subsheaves $F\subset E$ with 
$0<\rk(F)<\rk(E)$. 
\end{definition}

One easily proves
\begin{corollary}
If $E$ is a pure coherent sheaf $E$ of $\OOO_X$-modules on $X$ of dimension $d=\dim(X)$, then one
has the following chain of implications
$$E\textrm{ is $\mu$-stable}\Rightarrow E\textrm{ is stable}\Rightarrow E\textrm{ is semistable}\Rightarrow E\textrm{ is $\mu$-semistable}.$$
\end{corollary}
\subsection{Representable and corepresentable functors}

Let $\CCC$ be a category, $\CCC^{0}$ the opposite category, i.e. the category with the same objects and reversed arrows,
and let $\CCC'$ be the functor category whose objects are the functors $\CCC^{0}\ra\textrm{Sets}$ and whose
morphisms are the natural transformations between functors. The Yoneda Lemma states that the functor $\CCC\rar\CCC'$ 
which associates to $\textit{x}\in\Ob(\CCC)$ the functor $\underbar{\textit{x}}:y\mapsto \Mor_{\CCC}(y,\textit{x})$ embeds $\CCC$ as a full
subcategory into $\CCC'$. A functor in $\CCC'$ of the form $\underbar{\textit{x}}$ is said to be represented by the object $\textit{x}$.

Let $X$ be a projective scheme, $\CCC=\Sch/k$. For a fixed polynomial $P\in\QQ[z]$ define a functor
$$\MMM'_{X}(P):\CCC^{0}\ra\textrm{Sets}$$ as follows. If $S\in\Ob(\Sch/k)$ (that is $S$ is a scheme with a morphism $S\rar\Spec(k)$) , let $\MMM'(S)$ be the set of $S$-flat families $\FFF\rar X\times S$ of vector bundles on $X$ all of whose fibres have Hilbert polynomial $P$. And if $f:S'\rar S$ is a morphism in $(\Sch/k)$, let $\MMM'(f)$ be the map obtained
by pulling back sheaves via $f_X=id_X\times f:$
$$\MMM'_{X}(P)(f):\MMM'_{X}(P)(S)\rar\MMM'_{X}(P)(S'), [F]\rar[f^{*}_X F].$$

If $\FFF\in\MMM'_{X}(P)(S)$ is an $S$-flat family of vector bundles on $X$ with Hilbert polynomial $P$, and if $L$ is an arbitrary line bundle on $S$, then $\FFF\otimes p^{*}(L)$ is also an $S$-flat family of vector bundles on $X$ with a Hilbert polynomial, where $p$ is the natural projection from $X\times S$ to $S$. It is therefore reasonable to consider the quotient $\MMM_{X}(P)=\MMM'_{X}(P)/\sim$, where $\sim$ is the equivalence relation:
$$\FFF\sim\FFF'\iff\FFF\simeq\FFF'\otimes p^{*}L\textrm{ for some $L$ in $\Pic(S)$}.$$ 

\begin{definition}
A functor $\FFF\in\Ob(\CCC')$ is corepresented by $F\in\Ob(\CCC)$ if there is a $\CCC'$-morphism $\alpha:\FFF\rar\underbar{\textit{F}}$ such that any morphism $\alpha':\FFF\rar\underbar{\textit{F}}'$ factors through a unique morphism $\beta:\underbar{\textit{F}}\rar\underbar{\textit{F}}'$.

Assume now that $\CCC$ admits fiber products, then so does $\CCC'$. A functor $\FFF$ is universally corepresented by $\alpha:\FFF\rar\underbar{\textit{F}}$ if for any $T\in\Ob(\CCC)$ and any morphism $\phi:\underbar{\textit{T}}\rar\underbar{\textit{F}}$, the fibre product $\Gamma=\underbar{\textit{T}}\times_{\underbar{\textit{F}}}\FFF$ is corepresented by $T$. 
\end{definition}
\begin{definition}
A coarse moduli scheme of vector bundles on a polarized projective scheme $X$ over $k$ with Hilbert polynomial $P$ is a scheme $M_X(P)$
such that the functor $\MMM_X (P)=\MMM'_X (P)/\sim$ is universally corepresented by $M_X(P)$.
\end{definition}

\begin{definition}
$\FFF$ is represented by $F$ if $\alpha:\FFF\rar\underbar{\textit{F}}$ is a $\CCC'$-isomorphism.
We can rephrase that definition by saying that $F$ represents $\FFF$ if $\Mor_{\CCC}(\YYY, F)=\Mor_{\CCC'}(\underbar{\textit{Y}}, \FFF)$ for all $\YYY\in\Ob(\CCC)$.
\end{definition}
\begin{definition}
A fine moduli space of vector bundles on a polarized projective scheme $X$ over $k$ with Hilbert polynomial $P$ is a scheme $M_X(P)$ which represents the functor $\MMM_X (P)$. In this case, $\MMM_{X}(P)(M_{X}(P))$ contains a universal vector bundle $\UUU$ over $X\times M_{X}(P)$ with the following property: for any $S\in\Sch/k$ and any $\EEE\in\MMM_{X}(P)(S)$, there exists a unique morphism $\phi:S\rar M_{X}(P)$ such that $(id_X\times\phi)^{*}(\UUU)\simeq\EEE$.
\end{definition}

\begin{remark}
If a fine moduli space $M_X (P)$ exists, it is unique up to an isomorphism.
Nevertheless, in general, the functor $\MMM_X (P)$ is not representable.
In fact, there are very few classification problems for which a fine moduli scheme exists.
To get, at least, a coarse moduli scheme, we must somehow restrict the class of vector bundles that
we consider. In \cite{Ma-1} and \cite{Ma-2}, M. Murayama found an answer: (semi)stable vector bundles. 
\end{remark}
If we take families of (semi)stable locally free sheaves with respect to $H=\OOO_X (1)$ only, we get open subfunctors $(\MMM')^{ss}_{X}(P)\subset\MMM'_{X}(P)$, resp $\MMM'_{X}(P)^{s}\subset\MMM'_{X}(P)$ and $\MMM^{ss}_{X}(P)\subset\MMM_{X}(P),$ resp $\MMM^{s}_{X}(P)\subset\MMM_{X}(P)$, and $\MMM^{s}_{X}(P)$ is open in $\MMM^{ss}_{X}(P)$.

\subsection{Construction of moduli space}

A necessary condition for the existence of a coarse moduli space for the functor $\MMM^{ss}_{X}(P)$ as a scheme of finite type 
on $k$ is the boundedness of the family of all semistable vector bundles on $X$ with Hilbert polynomial $P$.
\begin{definition}
A family of isomorphism classes of coherent sheaves of $\OOO_X$-modules on $X$ is bounded if there exists a $k$-scheme $S$
of finite type and a coherent sheaf $F$ of $\OOO_{X\times S}$-modules on $X\times S$ such that the given family is contained
in the set $\{F_{s}\rvert \textrm{$s$ a closed point in $S$.}\}$
\end{definition}
To present $S$ and an $\OOO_{X\times S}$-module $F$ providing the boundedness for $\MMM^{ss}_{X}(P)$, we need the following definition.

\begin{definition}
Let $m$ be an integer. A coherent sheaf $F$ is said to be $m$-regular, if
$$H^{i}(X, F(m-i))=0\textrm{ for all $i>0$}.$$
\end{definition}
\begin{lemma}
For any semistable sheaf $F$ with Hilbert polynomial $P$, there is an integer $m$ such that F is $m$-regular.
\end{lemma}
\begin{proof}
Follows Serre's vanishing Theorem. 
\end{proof}
\begin{definition}
Let $(X,\OOO_X(1))$ be a polarized projective scheme over $k$, $S$ a $k$-scheme of finite type, $\CCC=(\Sch/S)$,
$\HHH$ a $S$-flat coherent sheaf of $\OOO_X$-modules with Hilbert polynomial $P$, then we define 
the functor $$\boldsymbol{\Quot_{X/S}(\HHH,P)}:\CCC^{0}\rar \textrm{ Sets}$$ as follows:
If $T$ is a $S$-scheme, then $\boldsymbol{\Quot_{X/S}(\HHH,P)}(T)$ is the set of $T$-flat coherent quotient sheaves $F$ of the sheaf $\HHH_T=\HHH\otimes_{\OOO_{S}}\OOO_{T}$ such that the fibers of $F$ over all the geometric points of the Grassmann projective $S$-scheme $\GGG^r$ have Hilbert polynomial $P$.
\end{definition}
\begin{theorem}
The functor $\boldsymbol{\Quot_{X/S}(\HHH,P)}$ defined above is represented by a projective $S$-scheme $\Quot_{X/S}(\HHH,P)$ with the
universal quotient sheaf $\UUU$.
\end{theorem}
\begin{proof}
See Theorem $2.2.4$ of \cite{H-L}.
\end{proof}

\begin{lemma}\label{bound}
As a family representing all the semistable sheaves from $\MMM^{ss}_X(P) (k)$, one can take the universal quotient sheaf $\UUU$ over $\Quot_{X/S}(\HHH,P)$, where $\HHH=k^{\oplus P(m)}\otimes\OOO_X (-m)$ and $m$ is such that $F$ is $m$-regular for all the semistable sheaves $F$  on $X$ with Hilbert polynomial $P$.
\end{lemma}
  
\begin{corollary}
The family of semistable sheaves with fixed Hilbert polynomial $P$ on a smooth projective variety $X$ is bounded or, in other words, the
functor $\MMM$ is bounded.
\end{corollary}
\begin{proof}
Follows from Lemma \ref{bound}.
\end{proof}
\begin{definition}
The $S$-equivalence classes are the same as Jordan classes. They are defined for sheaves (or vector bundles) on $X/k$ as the classes of
semistable sheaves with graded objects which are stable and having the same reduced Hilbert polynomial with respect to the Harder-Narasimhan filtration. Let $F$ and $F'$ be
semistable sheaves with filtration $(F_i)$ and $(F'_i)$, then $F$ is $S$-equivalent to $F'$ if and only if \\
$(1)$ $\rk(F)=\rk(F')$\\
$(2)$ The quotients $F_i/F_{i-1}\simeq F'_i/F'_{i-1}$ up to an appropriate permutation.\\
\end{definition}
\begin{theorem}
The functor $\MMM^{ss}_{X}(P)$ has a coarse moduli scheme $M^{ss}_X (P)$ which is quasi-projective over $k$, and the points of $M^{ss}_{X}(P)$ 
represent the $\textrm{S}$-equivalence classes of semistable sheaves with Hilbert polynomial $P$. There exists an open subscheme $M^{s}_X(P)$ of $M^{ss}_X(P)$ which is quasi-projective and whose points represent the isomorphism classes of stable sheaves with Hilbert polynomial $P$.
\end{theorem}
\begin{proof}
This is Theorem $1.21$ of \cite{Sim} 
\end{proof}

\begin{remark}
Stable vector bundles are a class of vector bundles with the property that families of stable bundles over $\Spec(K)\subset\Spec(R)$, where $R$ is
a discrete valuation ring with quotient field $K$, have at most one extension to families of stable bundles over $\Spec(R)$. Hence, by the valuative criterion, the moduli space is separated if it exists. In higher dimension, if we want to represent the functor by a projective moduli space, we have to consider not just (semi)stable vector bundles, but (semi)stable torsion-free sheaves.
\end{remark}

\begin{remark}
$(1)$ If a coarse moduli space exists for a given classification problem, then it is unique (up to an isomorphism).
$(2)$ A fine moduli space for a given classification problem is always a coarse moduli space for this problem, but, in general
not vice versa. In fact, there is no a priori reason why the map $$\Phi(S):\MMM^{s}(S)\rar\Hom(S, M^{s})$$
should be bijective for varieties $S$ other than a point.
\end{remark}
We refer to \cite{H-L} for general facts on the infinitesimal structure of the moduli space $M^{s}$.
Just let me recall that if $E$ is a stable vector bundle on $X$ with Hilbert polynomial $P$, represented by a point $[E]\in M^s$, then
the Zariski tangent space of $M^s$ at $[E]$ is given by $T_{[E]}M^{s}\simeq\Ext^{1}(E, E)$. If $\Ext^{2}(E, E)=0$,
then $M^s$ is smooth at $[E]$. In general, we have the following bounds:
$$\dim_k \Ext^{1}(E,E)\geq\dim_{[E]} M^s\geq\dim_k \Ext^{1}(E,E)-\dim_k\Ext^{2}(E, E).$$

We rely on Lemma \ref{bound} to describe the family of stable resp. semistable sheaves on $X$.
Thus, $F(m)$ is generated by global sections. If we set $V=k^{\oplus P(m)}$, $\HHH:=V\otimes\OOO_X(-m)$, then there exists a surjection $\rho:\HHH\rar F$ obtained by composing the canonical evaluation map $H^0 (F(m))\otimes\OOO_X(-m)\rar F$ with an isomorphism $V\rar H^0(F(m))$ . This defines a closed point $[\HHH\rar F]\in\Quot(\HHH, P)$, more precisely this point is contained in the open subscheme $R$ of all those coherent quotient sheaves $[\HHH\rar E]$, where $E$ is semistable and the induced map $H^{0}(\HHH(m))\simeq H^{0}(E(m))$ is an isomorphism.
The family of stable sheaves is parameterized by an open subscheme $R^{s}$ of $R$. 
The family of $\textrm{S}$-equivalence classes of semistable sheaves is parameterized by the quotient of $R$ by $\GL(V)$.\\
The next lemma relates the moduli problem to that of finding a quotient for the group action.
\begin{theorem}
If $R\rar M_X^{ss}(P)$ is a categorical quotient for the $GL(V)$-action, then $M_X^{ss}(P)$ corepresents the functor $\MMM_X^{ss}(P)$. Conversely,
if $M_X^{ss}(P)$ corepresents $\MMM_X^{ss}(P)$ then the morphism $R\rar M_X^{ss}(P)$, induced by the universal quotient module on $R\times X$, is
a categorical quotient . Similarly, $R^s\rar M^s_X(P)$ is a categorical quotient if and only if $M^s_X(P)$ corepresents $\MMM^s_X(P)$.
Therefore, we have $M_X^{ss}(P)=R//GL(V)$ and $M^s_X(P)=R^s//GL(V)$.
\end{theorem}
\begin{proof}
See Lemma $4.3.1$ of \cite{H-L}.
\end{proof}

\section{Moduli of connections}\label{Moduli of connections}

Basics for constructing moduli spaces of connections were developed by Simpson \cite{Sim} and Nitsure \cite{Ni}.
Nitsure constructed the coarse moduli space of integrable logarithmic connections with poles on a normal crossing divisor $D$ in
a smooth projective manifold. Simpson provided a more general approach covering not only regular integrable connections and
Nitsure's case of integrable logarithmic connections, but also Higgs bundles, Hitchin pairs, integrable connections along a foliation, Deligne's $\tau$-
connections and so on. Simpson handled all these objects on equal basis as $\Lambda$-modules for an appropriate sheaf of rings of differential operators $\Lambda$, which is just the standard sheaf of differential operators $\DDD_X$ for regular integrable connections, and
its associated graded $\gr\fatdot\DDD_X=\oplus_{m=0}^{\infty}S_{m}\TTT_{X}$ in the case of Higgs bundles, where $S_m\TTT_X$ denotes the $m$-th symmetric power of the tangent bundle
on~$X$.

\subsection{Sheaf $\boldsymbol\Lambda$ of rings of differential operators}

We will recall Simpson's definition of $\Lambda$. Let $X$ be a scheme of finite type over $k$, an algebraically closed field of characteristic $0$.
Then a sheaf of rings of differential operators on $X$ is a sheaf $\Lambda$ of associative rings with unity together with a filtration by subsheaves of
abelian groups $\Lambda_0\subset\Lambda_1\subset\Lambda_2\subset\dots$ satisfying the following properties:\\
$(1)$ $\Lambda_0=\OOO_X$, $\Lambda_i\cdot\Lambda_j\subset\Lambda_{i+j}$, and $\Lambda=\cup_{i=0}^{\infty}\Lambda_i$.
In particular, $\Lambda$ and each $\Lambda_i$ are $\OOO_X$-bimodules.\\
$(2)$ The image of the constant sheaf $\CC_X$ in $\OOO_X=\Lambda_0$ is contained in the center of $\Lambda$.\\
$(3)$ The left and the right structures of $\OOO_X$-modules on the $i$-th graded piece $\gr_i(\Lambda)=\Lambda_i/\Lambda_{i-1}$
are equal.\\
$(4)$ The sheaves $\gr_i(\Lambda)$ are $\OOO_X$-coherent.\\
$(5)$ The $\OOO_X$-algebra $\gr\fatdot(\Lambda)$ is generated by its component $\gr_1(\Lambda)$ of degree $1$.\\
$(6)$ There is a left $\OOO_X$-linear map $\iota:\gr_1(\Lambda)\rar\Lambda_1$ providing a splitting of the triple 
$0\rar\Lambda_0\rar\Lambda_1\rar\gr_1 (\Lambda)\rar 0$.\\
$(7)$ $\gr\fatdot(\Lambda)$ is the symmetric algebra $S\fatdot(\gr_1(\Lambda))$ over $\gr_1(\Lambda)$.

As we have mentioned above, $\Lambda=\DDD_X$ is an example. If $X$ is smooth, $\EEE$ a vector bundle and $\nabla$ an integrable connection over $X$, then we can consider $\EEE$ as a left $\Lambda$-module, setting $v\cdot s=\nabla_v (s)$ for any local vector field $v$ and any local section s of $\EEE$.
To include the case of non-integrable connections, we will replace the axiom $(7)$ by the following one\\
$(7')$ $\gr.(\Lambda)$ is the tensor algebra $T\fatdot(\gr_1(\Lambda))=\oplus_{i=0}^{\infty} T_i(\gr_1(\Lambda))$ over $\gr_1(\Lambda)$, where 
$T_i(V)=V^{\otimes i}$ for any module $V$.

Simpson's construction of moduli spaces of semistable $\Lambda$-modules works in all details whether $\Lambda$ satisfies the axioms $(1)-(7)$ or
$(1)-(6), (7')$. We will specialize it to the case of meromorphic connections as in Chapter $2$, and will briefly describe the steps of the construction of the moduli space. In  the sequel, $\Lambda$ satisfies the axioms $(7)$ or $(7')$  depending on whether we are working with integrable or arbitrary connections.

Let $X$ be a smooth variety over $k$ and $\boldsymbol{\DDD}_X$ the sheaf of non commutative rings of differential operators on $X$. It can be defined as follows. As a left $\OOO_X$-module, it is just the tensor algebra $T\fatdot({\TTT}_X)$ over the tangent bundle. To determine the multiplicative structure on it, it suffices to define the products $v\cdot f$, where $v\in\TTT_{X,p}$, $f\in\OOO_{X,p}$, $p\in X$. We set $$v\cdot f=fv+v(f),$$ where $v(f)$ denotes the derivative of $f$ in the direction of $v$. This rule allows us to transform the product of two elements of $T\fatdot(\TTT_X)$
$$f v_1\otimes\dots\otimes v_r\cdot g w_1\otimes\dots\otimes w_s$$ into an element of $T.(\TTT_X)$ in a finite number of steps.

Now, let us fix an effective divisor $D$ on $X$. We define $\Lambda$ as the subsheaf of rings in $\boldsymbol{\DDD}_X$ generated by $\OOO_X$ and
$\TTT_X ( -D)$, the latter sheaf being considered as the subsheaf of $\TTT_X\subset\boldsymbol{\DDD}_{X1}$ consisting of vector fields vanishing on $D$.
Then for a rank $r$ connection $(\EEE,\nabla)$ with divisor of poles $D$, we endow $\EEE$ with a structure of a left $\Lambda$-module by setting
$$v\cdot s=\nabla_vs=v\lrcorner\nabla s$$ for any $v\in\TTT_{X,p}, s\in\EEE_p, p\in X$. Conversely, we can completely recover $\nabla$ from a structure of a $\Lambda$-module in applying the above formula to the vector field $v$ ranging over some basis of $\TTT_{X,p}$ as an $\OOO_{X,p}$-module.
In the sequel of this section, we will think of connections $(\EEE,\nabla)$ with divisor of poles $D$ as left $\Lambda$-modules that are
locally free of rank $r$ when considered as $\OOO_X$-modules. For intermediate steps of this construction, we also need to consider coherent $\OOO_X$-modules with a structure of a left $\Lambda$-module.

From now on, $\Lambda$ is any sheaf of rings of differential operators satisfying either axioms $(1)-(7)$ or $(1)-(6), (7')$.
The moduli space of semistable $\Lambda$-modules that we are going to construct is interpreted as the moduli space of a class of connections
in the following cases.\\
$(a)$ $\Lambda\subset\boldsymbol{\DDD}_X$, generated by $\OOO_X$ and $\TTT_X(-D)$ and satisfying axiom $(7')$, correspond to meromorphic connections with fixed divisor of poles $D$.\\
$(b)$ $\Lambda\subset\DDD_X$, generated by $\OOO_X$ and $\TTT_X(-D)$ and satisfying axiom $(7)$, correspond to integrable meromorphic connections with fixed divisor of poles $D$.\\
$(c)$  $\Lambda\subset\boldsymbol{\DDD}_X$, generated by $\OOO_X$ and $\TTT_X<D>=\TTT_X (\log D)$ and satisfying axiom $(7')$, correspond to logarithmic connections with a simple normal crossing divisor of poles $D$.\\
$(d)$  $\Lambda\subset\DDD_X$, generated by $\OOO_X$ and $\TTT_X<D>=\TTT_X (\log D)$ and satisfying axiom $(7)$, correspond to integrable logarithmic connections with a simple normal crossing divisor of poles $D$.

The sheaf $\TTT_X <D>$ is dual to $\Omega^1_X (\log D)$ and can be defined as the subsheaf of $\TTT_X$ preserving the ideal subsheaf $\III_D$ of $D$ in $X$.

\subsection{Moduli space of semistable $\boldsymbol\Lambda$-modules}

To speak about quasi-projective moduli spaces, we have to start with defining the notions
of stability and semistability.

\begin{definition}
Let $X$ be a smooth projective variety with a very ample sheaf $\OOO_X(1)$. A coherent $\OOO_X$-module $\EEE$ of rank $r>0$ endowed with
a structure of a left $\Lambda$-module is called a semistable $\Lambda$-module of rank $r$, if it is torsion free and for any
$\Lambda$-submodule $\FFF\neq0$ of $\EEE$,
$$\frac{P(\FFF,n)}{\rk(\FFF)}\leq\frac{P(\EEE,n)}{\rk(\EEE)}\ \forall n>>0,$$ where $P(\FFF,n)=\chi(\FFF,n)$ denotes the Hilbert polynomial
of $\FFF$. If the inequality is strict for all $\FFF$, then $\EEE$ is called a stable $\Lambda$-module. 
\end{definition}

The following lemmas are crucials for the boundedness of the family of all semistable $\Lambda$-modules of rank $r$ with fixed Hilbert polynomial $P$
on $X$.
\begin{lemma}
Let $\EEE$ be a semistable $\Lambda$-module on a smooth projective variety $X$ with a very ample sheaf $\OOO_X(1)$. Let $\GGG$ be the saturation
of the image of the morphism $\Lambda_r\otimes\FFF\rar\EEE$ for any $\OOO_X$-submodule $\FFF$, then $\GGG$ is a $\Lambda$-submodule of $\EEE$.
\end{lemma}
\begin{proof}
See the proof of Lemma $3.2$ of \cite{Sim}, which works perfectly for our definition of $\Lambda$.
\end{proof}
\begin{lemma}\label{boundi}
Let $m$ be an integer such that $gr_1(\Lambda)\otimes\OOO_X(m)$ is generated by global sections. Then for any semistable $\Lambda$-module $\EEE$
of rank $r$, and any $\OOO_X$-submodule $\FFF\neq 0$, we have $\mu(\FFF)\leq\mu(\EEE)+mr$, where $\mu(\FFF)$ denotes the slope of $\FFF$,
$\mu(\FFF)=\frac{\deg_{\OOO_X(1)}c_1(\FFF)}{\rk(\FFF)}$.
\end{lemma}
\begin{proof}
See the proof of Lemma $3.3$ of \cite{Sim}, which works perfectly for our definition of $\Lambda$.
\end{proof}
\begin{corollary}
The set of semistable $\Lambda$-modules on $X$ with given Hilbert polynomial $P$ is bounded.
\end{corollary}
\begin{proof}
We remark that semistable $\Rightarrow\mu$-semistable and refer to the proof of Corollary $3.4$ of \cite{Sim}.
\end{proof}
The following assertion realizes the boundedness property for semistable $\Lambda$-modules with given Hilbert polynomial $P$:
it provides a scheme parametrizing all of them. 
\begin{theorem}\label{boundness}
For fixed $P$, there exists $N_0\in\NN$ depending on $\Lambda$ and $P$ such that for any $N\geq N_0$ and any $S$-flat semistable $\Lambda$-module
$\EEE$ with Hilbert polynomial $P$ on $X$ such that $\forall s\in S$, we have $H^i(X,\EEE_s(N))=0$ if $i>0$, $\dim H^0(X,\EEE_s(N))=P(N)$ and $\EEE_s(N)$ is generated by global sections.

Pick up any $N\geq N_0$. Then the functor which associates to each $k$-scheme $S$ the set of isomorphism classes of pairs $(\EEE,\alpha)$,
where $\EEE$ is a semistable $\Lambda$-module with Hilbert polynomial $P$ on $X_S=X\times S$ and $\alpha$ is an isomorphism
$\OOO_S^{P(N)}\rar H^0(X_S/S,\EEE(N))$, is represented by a quasi-projective scheme $Q$ over $k$.
\end{theorem}
\begin{proof}
See Corollary $3.6$ and Theorem $3.8$ of \cite{Sim}.
\end{proof}
The scheme $Q$ is constucted in Theorem \ref{boundness} in several steps. First, take the Grothendieck $\Quot$ scheme 
$\tilde{Q}=\Quot^P(\OOO_X^{P(N)}(-N))$ parameterizing the quotients $\OOO_X^{P(N)}(-N)\rar\EEE\rar0$ with Hilbert polynomial $P$.
Over $\tilde{Q}$, one considers the family $\tilde{\tilde{Q}}$ of morphisms $\Lambda_1\otimes_{\OOO_X}\EEE\rar\EEE$ defining on the quotients
$\EEE$ structures of $\Lambda$-modules. This is a family with affine fibers (e.g. affines spaces in the case $(a)$ of non-integrable connections).
And finally, $Q$ is the open subscheme of $\tilde{\tilde{Q}}$ parameterizing the $\Lambda$-modules which are semistable.

Let now $\MMM(\Lambda,P)$ denote the functor on schemes over $k$ which associates to a $k$-scheme $S$ the set of isomorphism classes of
semistable $\Lambda$-modules with Hilbert polynomial $P$. We are now ready to construct the moduli space for this functor as a 
GIT quotient.
\begin{theorem}\label{modss}
Under the hypotheses and in the notation of Theorem \ref{boundness}, $Q$ is invariant under $G=SL(P(N))$ and carries a
G-linearized very ample line bundle $\LLL$ such that all the points of $Q$ are $\LLL$-semistable.

Let $M(\Lambda,P)=Q//G$ be the GIT quotient.
Then $M(\Lambda,P)$ universally corepresents $\MMM(\Lambda,P)$.
The following properties hold:\\
$(1)$ $M(\Lambda,P)$ is  a quasi-projective variety.\\
$(2)$ The geometric points of $M(\Lambda,P)$ represent $S$-equivalence classes of semistables $\Lambda$-modules with Hilbert polynomial $P$.\\
$(3)$ The closed orbits are in $1$-to-$1$ correspondence with the semisimple objects.\\
$(4)$ The geometric points of the open set $M^s(\Lambda,P)=Q^s//G$, where $Q^s$ parameterizes stable $\Lambda$-modules,
are in $1$-to-$1$ correspondence with the isomorphism classes of stable $\Lambda$-modules with Hilbert polynomial $P$.  
\end{theorem}
\begin{proof}
This is Theorem $4.7$ of \cite{Sim}.
\end{proof}
Note that the notion of $S$-equivalence and semisimple objects are defined exactly as in the case of moduli of sheaves.
Namely, any semistable $\Lambda$-module $\EEE$ has a Harder-Narasimhan filtration
$$\EEE_0=0\subset\EEE_1\subset\dots\subset\EEE_t=\EEE$$ with the property that all the factors $gr_i\EEE=\EEE_i/\EEE_{i-1}$ are
stable $\Lambda$-modules with the same reduced Hilbert polynomial, equal to the reduced Hilbert polynomial
$\frac{P(\EEE,n)}{\rk{\EEE}}$ of $\EEE$.\\
Two semistable $\Lambda$-modules are called $S$-equivalent if the associated graded objects of their Harder-Narasimhan filtrations
are isomorphic. Further, $\EEE$ is semisimple if $\EEE\simeq gr(\EEE)$.

Finally, if we assert the additional restriction that our $\Lambda$-modules are locally free as $\OOO_X$-modules, we will obtain the
open subvarieties $M^0_X(\Lambda,P)\subset M(\Lambda,P)$ and $M^{0,s}(\Lambda,P)\subset M^s(\Lambda,P)$, moduli spaces of vector bundles with a 
structure of a $\Lambda$-module

\section{Luna slice theorem}
In \cite{Machu-2}, we have constructed the formal versal deformations, called also formal Kuranishi spaces, for $4$ types
of connections: all the connections with fixed divisor of poles $D$, integrable ones,
integrable logarithmic connections and integrable logarithmic ones with a parabolic structure 
over $D$. It is quite easy to see that our formal Kuranishi spaces lift to germs of complex analytic spaces, that is,
our formal series have nonzero radius of convergence. The Luna slice theorem allows us to go further and produce an affine
scheme with a marked point whose germ at the marked point is the base of a versal deformation. The Luna slice theorem is
stated in the general framework of a reductive algebraic group acting on a $k$-scheme $X$ of finite type.
\begin{definition}\label{def}
Let $G$ be an affine algebraic group over $k$ acting on a $k$-scheme $X$. A morphism $\phi:X\rar Y$ is a good quotient, if\\
$(1)$ $\phi$ is affine, surjective and open.\\
$(2)$ The natural homomorphism $\OOO_Y\rar(\phi_{*}\OOO_X)^G$ is an isomorphism.\\
$(3)$ If $W$ is an invariant closed subset of $X$, then its image $\phi(W)$ is also a closed subset of $Y$.
If $W_1$ and $W_2$ are disjoint invariant closed subsets of $X$, then $\phi(W_1)\cap\phi(W_2)=\emptyset$.
\end{definition}

\begin{definition}
In the situation of Definition \ref{def}, $\phi$ is a geometric quotient, if it is a good quotient, and the geometric fibers of $\phi$ are the orbits of geometric points of $X$. We will denote a good quotient of $X$, if it exists, by $X//G$.\end{definition}

In the constructions of moduli spaces, described in the previous sections, the quotient of the semistable locus of the
$\Quot$ scheme by $SL(P(N))$ is a good quotient, and the quotient of the stable one is a geometric quotient.
\begin{definition}
Let $G$ be an algebraic group , $H\subset G$ an algebraic subgroup, $V$ a $k$-scheme of finite type
with an action of $H$. Make $H$ act on the product $G\times V$ according to the rule
$$h:(g,v)\mapsto(gh^{-1},hv).$$
Then there exists a geometric quotient $G\times V//H$ such that the natural map $g:G\times V\rar G\times V//H$ is
a $H$-principal bundle. We  denote $G\times V//H$ by $G\times^{H} V$. It has a natural (left) action of $G$, and
we say that $G\times^{H} V$ is obtained from $V$ by extending the action from $H$ to $G$.
\end{definition}
\begin{definition}
Let $G$ be an affine algebraic group acting on a $k$-scheme $X$ of finite type, $x_0\in X$, $\OOO(x_0)=G\cdot x_0$ the orbit of $x_0$.
A normal slice to $\OOO(x_0)$ at $x_0$ is an affine scheme $S\subset X$ with the following properties:\\
$(1)$ $x_0\in S$ and $S$ is invariant under the action of $G_{x_0}=\Stab_G(x_0)$, the stabilizer of $x_0$ in $G$.\\
$(2)$ The natural morphism $\phi:G\times^{G_{x_0}} S\rar X$ has an open image and is \'etale over its image.
\end{definition}
\begin{Luna Slice Theorem}[\cite{Ln}]
Let $X$ be a $k$-scheme of finite type, $G$ a reductive algebraic group acting on $X$, and let $\pi:X\rar X//G$ be a good quotient.
Let $x\in X$ be a point such that $\OOO(x)$ is closed. Then there exists a normal slice $S$ to $\OOO (x)$ at $x$ and the stabilizer $G_x$ is a reductive algebraic group so that there exists a good quotient  $S//G_x$. Moreover, the induced morphism of good quotients $S//G_x\rar X//G$ has an affine open image and is \'etale over the image. Furthermore the following diagram is commutative.
\begin{displaymath}
\xymatrix{S\times^{G_x} G\ar[r]\ar[d] & X\ar[d]\\ S//G_x\ar[r] & X//G}
\end{displaymath}
If $X$ is normal (resp. smooth) at $x$, then $S$ can be also taken normal (resp. smooth).
\end{Luna Slice Theorem} 
\begin{proof}
See \cite{Ln}.
\end{proof}

The fact that the Luna normal slices are versal deformations is known for moduli of sheaves and was used by several authors for 
computing the local structure of the moduli space of sheaves at a strictly semistable point
\cite{O'G}, \cite{Dr}, \cite{L-S}, \cite{M-T}.

Now, we will prove that a similar property holds for moduli of connections.
\begin{theorem}
Let $\Lambda$ be as in one of the cases $(a)-(d)$ of Sect. \ref{Moduli of connections}, and set the hypotheses and the notation as in Theorems \ref{boundness},\ref{modss} of Sect. \ref{Moduli of connections}. Let $\EEE$ be a polystable $\Lambda$-module with Hilbert polynomial $P$, 
$z=[\EEE]$ the corresponding point of $Q$. Assume that $\EEE$ is locally free as an $\OOO_X$-module, then the orbit $\OOO(z)=G\cdot z$ is
closed, so that there is a normal slice $V$ at $z$. Let $\boldsymbol{\EEE}$ be the restriction to $X\times V$ of the tautological quotient $\Lambda$-module
over $Q$. Then the couple $(V,\boldsymbol{\EEE})$ is a versal deformation of $\EEE$.  
\end{theorem} 
\begin{proof}
The versality of $(V,\boldsymbol{\EEE})$ means that the following two properties are verified:\\
$(1)$ Any (flat) deformation $\FFF$ of $\EEE$ over any $k$-scheme $S$ is induced from $(V,\boldsymbol{\EEE})$ via some morphism $S\rar V$.\\
$(2)$ The Kodaira-Spencer map 
$$\kappa:T_z V\rar\Ext^1_{\Lambda}(\EEE,\EEE)$$ is an isomorphism.

The property $(1)$ follows easily from the universal property of the Quot scheme.
We will prove $(2)$ for the case $(a)$; the other cases are treated similarly.
First note that a structure of a $\Lambda$-module is the same as a connection $\nabla$ with divisor
of poles $D$, so $\Ext^1_{\Lambda}(\EEE,\EEE)=\HH^1(\CCC^{\fatdot})$, where 
\begin{equation*} \CCC^{\fatdot}=[\mathcal{E}nd_{\OOO_X}(\EEE)\xymatrix@1{\ar[r]^{\nabla}&}\mathcal{E}nd_{\OOO_X}(\EEE)\otimes_{\OOO_X}\Omega^1_X (D)],\end{equation*}
is the complex introduced in Theorem $2.9$. Thus, we have the following exact triple:
\begin{equation}
0\rar H^0(\mathcal{E}nd_{\OOO_X}(\EEE)\otimes_{\OOO_X}\Omega^1_X (D))/\nabla(H^0(\mathcal{E}nd_{\OOO_X}(\EEE))\rar\Ext^1_{\Lambda}(\EEE,\EEE)\rar\Ext^1_{\OOO_X}(\EEE,\EEE)
\rar 0.
\end{equation}
Now, a tangent vector from $T_z Q$ is an infinitesimal deformation of the quotient map $\OOO_X(-N)^{P(N)}\rar\EEE\rar 0,$
or equivalently, of the exact triple
\begin{equation}\label{inf}
0\rar\KKK\rar\OOO_X(-N)^{P(N)}\rar\EEE\rar 0,\end{equation}
followed by an infinitesimal deformation of the connection $\nabla$ on $\EEE$.
The Quot-scheme parametrizing the triple \eqref{inf} was denoted $\tilde{Q}$ (Sect. $2$, before Theorem \ref{modss}), and
$$T_z\tilde{Q}=\Hom_{\OOO_X}(\KKK,\EEE).$$
Next, the deformations of $\nabla$ over the algebra of dual numbers $k[\epsilon]/\epsilon^2$ that fix $\EEE$ are
obviously parametrized by $H^0(\mathcal{E}nd_{\OOO_X}(\EEE)\otimes\Omega^1_X (D))$; two such deformations are isomorphic
if they differ by an element of $\nabla(H^0(\mathcal{E}nd_{\OOO_X}(\EEE)))$, so we obtain the exact triple
\begin{equation}\label{na}
0\rar H^0(\mathcal{E}nd_{\OOO_X}(\EEE)\otimes_{\OOO_X}\Omega^1_X (D))/\nabla(H^0(\mathcal{E}nd_{\OOO_X}(\EEE))\rar T_z Q\rar\Hom_{\OOO_X}(\KKK,\EEE)\rar 0.
\end{equation}
The Kodaira-Spencer map $\kappa$ induces a morphism of triples \eqref{inf}, \eqref{na}, which is the identity on the left hand side of
the triples. The right hand side counterpart $\bar{\kappa}$ of $\kappa$ can be identified, as in the proof of
Prop. $1.2.3$ of \cite{O'G}, with the homomorphism in the long exact sequence
\begin{equation*} 0\ra\Hom(\EEE,\EEE)\rar\Hom(\OOO_X(-N)^{P(N)},\EEE)\xymatrix@1{\ar[r]^{\alpha}&}\Hom(\KKK,\EEE)\xymatrix@1{\ar[r]^{\bar{\kappa}}&}\Ext^1(\EEE,\EEE)\rar\dots,\end{equation*}
obtained by applying $\Hom(\cdot,\EEE)$ to \eqref{inf}.
In loc. cit., it is proved that $\bar{\kappa}$ is surjective.
Hence, $\kappa$ is also surjective.
\end{proof}
\begin{corollary}\label{coro Luna}
Let $(\EEE,\nabla)$ be a connection arising in one of the $4$ cases $(a)-(d)$ of Section $2$ with locally free sheaf $\EEE$, and 
let $M=M_X (\Lambda,P)$ be the corresponding moduli space. We denote by $(K,0)$ the formal Kuranishi space of $(\EEE,\nabla)$ constructed
in Chapter $2$. Let $z\in M$ be the point representing $(\EEE,\nabla)$ in $M$, and $H=\Aut(\EEE,\nabla)$. Then
$$(K,0)//H=(M,z).$$
\end{corollary}
We will use this corollary in the next section to describe the local structure of moduli spaces of connections
in some examples.
\section{Examples}
Let $X$ be a curve, $D$ an effective divisor on $X$, $r\in\NN$ and $d\in\ZZ$. We will use the following notation:\\
$\CCC_X(r,d;D)$, the moduli space of semistable pairs $(\EEE,\nabla)$, where $\EEE$ is a vector bundle on $X$ of rank $r$ and degree $d$,
and $\nabla$ is a meromorphic connection on $\EEE$ with fixed divisor of poles $D$.\\
$M_X(r,d)$, the moduli space of semistable vector bundles of rank $r$ and degree $d$ on $X$.\\
$M_X^s(r,d)$, the locus of stable vector bundles in $M_X(r,d)$.\\
$\CCC_X^s(r,d;D)$, the locus of stable pairs $(\EEE,\nabla)$ in $\CCC_X(r,d;D)$.\\
$\CCC_X^0(r,d;D)$, the locus of pairs $(\EEE,\nabla)\in\CCC_X(r,d;D)$ with stable $\EEE$.\\
$\tilde{\CCC}_X(r,d;D)$, the locus of pairs $(\EEE,\nabla)\in\CCC_X(r,d;D)$ with semistable $\EEE$.\\
$\tilde{\CCC}_X^s(r,d;D)$, the locus of pairs $(\EEE,\nabla)\in\CCC_X^s(r,d;D)$ with semistable $\EEE$.

In the sequel, we will determine some of the relations between these moduli loci and produce examples of 
computing their local structure, based on the Luna slice theorem and Corollary \ref{coro Luna}.

\subsection{The case when the underlying vector bundle is stable}
Let $(\EEE,\nabla)\in\CCC^0(r,d;D)$. Then $\EEE$ is automatically stable.
The map of forgetting the second component of a pair is a well-defined morphism
$\pi:\CCC^0(r,d;D)\rar M^s(r,d).$ Moreover, given two connections $\nabla,\nabla'$ with
divisor of poles $D$ on the same vector bundle $\EEE$, we have $\nabla-\nabla'\in H^0(X,\mathcal{E}nd(\EEE)\otimes\Omega_X^1(D)).$
Thus the fiber of $\pi$ over a point $[\EEE]\in M^s(r,d)$, if nonempty, is the affine space $H^0(X,\mathcal{E}nd(\EEE)\otimes\Omega_X^1(D))$.
To determine the image $\im\pi\subset M^s(r,d)$, recall that $\EEE$ admits a connection with fixed divisor of poles $D$ if and
only if the Atiyah class $\At^D(\EEE)$ vanishes in $H^1(X,\mathcal{E}nd(\EEE)\otimes\Omega_X^1(D))$. By Serre duality,
$$H^1(X,\mathcal{E}nd(\EEE)\otimes\Omega_X^1(D))^*=H^0(X,\mathcal{E}nd(\EEE)\otimes\OOO_X(-D)),$$ and by stability,
$h^0(X,\mathcal{E}nd(\EEE))=1$, the only global endomorphisms of $\EEE$ being the homotheties. Thus $H^1(X,\mathcal{E}nd(\EEE)\otimes\Omega_X^1(D))=0$
whenever $D>0$, and $h^1(X,\mathcal{E}nd(\EEE)\otimes\Omega_X^1)=1$. $\EEE$ may fail to be in the image of $\pi$ only when $D=0$. In this case,
we refer to a theorem of Atiyah:
\begin{theorem}[\cite {A0}]
A holomorphic vector bundle $\EEE$ on a curve $X$ admits a holomorphic connection
if and only if $\EEE$ is semistable of degree $0$.
\end{theorem}

Note also that the dimension of fibres of $\pi$ is constant, as follows from the Riemann-Roch theorem.
This implies that $\pi$ is a locally trivial fiber bundle with fiber an affine space and with structure group the
affine group. The local triviality follows only in the classical and the \'etale topologies, for $\CCC^0_X(r,d;D)$ is
not a fine moduli space in general and possesses universal connections $(\boldsymbol{\EEE},\boldsymbol{\nabla})$ only in the classical or
\'etale topology.
Summarizing the above, we state the following theorem.
\begin{theorem}\label{struc}
Let $X$ be  curve of genus $g\geq 1$, $D$ an effective divisor on $X$, $r\in\NN$ and $d\in\ZZ$. Then the following
assertions hold:

$(i)$ Assume $D=0$. Then $\CCC_X^0(r,d;0)=\varnothing$ if $d\neq 0$, and $\CCC_X^0(r,0;0)=M_X^s(r,0)=\varnothing$ if $r>1$ and
$g=1$. In the case when either $g\geq 2$, or $r=g=1$, the map $\pi:\CCC_X^0(r,0;0)\rar M_X^s(r,0)$ is an affine bundle, locally trivial
in the \'etale topology, with fiber $\CC^{r^2(g-1)+1}$.

$(ii)$ Assume $D>0$ and $M_X^s(r,d)\neq\varnothing$, that is either $g\geq 2$, or $g=1$ and $g.c.d.(r,d)=1$. Then the map
$\pi:\CCC_X^0(r,d;D)\rar M_X^s(r,d)$ is an affine bundle, locally trivial
in the \'etale topology, with fiber $\CC^{r^2(g-1+\deg(D))}$.
\end{theorem}

\subsection{The case $g=1, r=2, d=0$: local structure of $\tilde{\CCC}_X(2,0;0)$ at the most degenerate point}\label{sect. 5.2}
Let $X$ be an elliptic curve, $\nabla=d$ the trivial connection, equal to the de Rham differential.
We refer to $(\EEE,\nabla)$ as the most degenerate point of $\tilde{\CCC}_X(2,0;0)$, the moduli space of 
regular connections on semistable rank-$2$ vector bundles of degree $0$ over an elliptic curve.
Let $\CCC^{\fatdot}$ be the complex of sheaves 
$$\mathcal{E}nd(\EEE)\xymatrix@1{\ar[r]^{\nabla}&}\mathcal{E}nd(\EEE)
\otimes\Omega^{1}_X.$$
Consider the long exact sequence
\arraycolsep=0.5ex
\begin{eqnarray}\label{lgexs}
0&\lra&\HH^0(X,\CCC^{\fatdot})\lra H^{0}(X,\mathcal{E}nd(\EEE))\xymatrix@1{\ar[r]^{d_1}&}H^{0}(X,\mathcal{E}nd(\EEE)
\otimes\Omega^{1}_X)\\ \nonumber
&\lra&\HH^{1}(X,\CCC^{\fatdot})\lra H^{1}(X,\mathcal{E}nd(\EEE)) \xymatrix@1{\ar[r]^{d_1}&}H^{1}(X,\mathcal{E}nd(\EEE)\otimes\Omega^{1}_X )\\ \nonumber
&\lra&\HH^{2}(X,\CCC^{\fatdot})\lra H^{2}(X,\mathcal{E}nd(\EEE))=0 
\nonumber\end{eqnarray}
coming from the spectral sequence 
$E_1^{p,q}=H^{q}(\CCC^p)\Rightarrow\HH^{p+q}(\CCC^{\bullet}),$ supported on two vertical strings
$p=0$ and $p=1$. The maps $d_1$ are commutators $B\mapsto [A,B]=
A\circ B-B\circ A$, where $\circ$ denotes the Yoneda composition.
The first map $d_1$ is zero, for it is induced by $\nabla_{\mathcal{E}nd(\EEE)}=d$ (de Rham differential),
and $H^0(X,\mathcal{E}nd(\EEE))=\CC^4\id_{\EEE}$, and the second map $d_1$ is the adjoint of the first one with
respect to the Serre duality, so it is zero, too.
We conclude that $\HH^2(X,\CCC^{\fatdot})\simeq H^1(X,\mathcal{E}nd(\EEE)\otimes\Omega^{1}_X )\simeq H^1(X,\OOO_X^{\oplus 4}\otimes\Omega_X^1)\simeq\CC^4,$
and \begin{eqnarray}\label{exac}  0\rar H^{0}(X,\mathcal{E}nd(\EEE)
\otimes\Omega^{1}_X)\rar\HH^{1}(X,\CCC^{\fatdot})\rar H^{1}(X,\mathcal{E}nd(\EEE)\rar 0\end{eqnarray}
is an exact triple, so that $\HH^{1}(X,\CCC^{\fatdot})\simeq\CC^8$.
We can represent the elements of $\HH^{1}(X,\CCC^{\fatdot})$ as the pairs $(A,a)$, where $a\in H^1(X,\mathcal{E}nd(\EEE))$ and $A\in H^0(X,\mathcal{E}nd(\EEE)\otimes\Omega^{1}_X).$ As $\Omega^1_X\simeq\OOO_X$, both $a,A$ can be represented by $2\times 2$-matrices,
which we will write down in the following form:
$$
A=\left( \begin{array}{cc} x & x_{12} \\ x_{21} & -x \end{array}\right)\ ,\ \ \  a=\left( \begin{array}{cc} y & y_{12} \\ y_{21} & -y
\end{array}\right)\delta,
$$
where $\delta$ is a generator of $H^1(X,\Omega^1_X)$.

The identification of $\HH^{1}(X,\CCC^{\fatdot})$ with $H^{0}(X,\mathcal{E}nd(\EEE)
\otimes\Omega^{1}_X) \oplus H^{1}(X,\mathcal{E}nd(\EEE))$ assumes that some cross-section of the epimorphism 
$\HH^{1}(X,\CCC^{\fatdot})\rar H^{1}(X,\mathcal{E}nd(\EEE))$ is fixed. In the case under consideration, \eqref{exac}
is just the standard Dolbeault cohomology triple
\begin{eqnarray}\label{Dolb}0\rar H^0(X,\Omega^1_X)\rar H^1(X,\CC)=\HH^1(X,\Omega^{\fatdot}_X)\rar H^1(X,\OOO_X)\rar 0\end{eqnarray} tensored by $\End(\CC^2)=H^0(X,\mathcal{E}nd(\EEE))$.
Thus $H^1(X,\OOO_X)$ can be identified with $H^{0,1}(X,\CC)=\overline{H^{1,0}(X,\CC)}$ in $H^1(X,\CC)$.
This fixes the choice of a splitting of (\ref{Dolb}) and hence of (\ref{exac}). The first obstruction map is then given by the commutator:
$$\ob_2: (A,a)\mapsto [A,a].$$ As the identity matrix commutes with any other matrix, we will split the summands $x_0\id$ and $y_0\id$
out of $A,a$. In fact, it is easy to prove that all the components of the Kuranishi map are independent of $x_0, y_0$, so that the Kuranishi space is in this case of the form
$K\simeq\CC^2\times\tilde{K}$. Here $\CC^2$ has $x_0,y_0$ as coordinates and $\tilde{K}$ is the zero locus of the reduced Kuranishi map
$\tilde{ob}=\tilde{\ob}_2+\tilde{\ob}_3+\dots:\CC^6\rar\CC^3$, where $\CC^6$ has $(x,x_{12},x_{21},y,y_{12},y_{21})$ as coordinates,
and $\CC^3$ in the target of $\tilde{\ob}$ is the traceless part of $H^1(X, \mathcal{E}nd(\EEE)\otimes\Omega^1_X)$. Computing
in coordinates, we obtain the initial term:
\begin{displaymath}\tilde{\ob}_2: 
\Big(A=\left( \begin{array}{cc} x & x_{12} \\ x_{21} & -x \end{array}\right)\ , a=\left( \begin{array}{cc} y & y_{12} \\ y_{21} & -y
\end{array}\right)\delta\Big)\mapsto [A,a]=\left( \begin{array}{cc} q_1 & q_2 \\ q_3 & -q_1\end{array}\right)\delta,\end{displaymath}
where $q_1=x_{12}y_{21}-x_{21}y_{12}, q_2=2xy_{12}-2x_{12}y, q_3=2x_{21}y-2y_{21}x$.
Thus $\tilde{\ob}_2^{-1}(0)$ is a complete intersection of three quadrics of rank $4$ in  $\CC^6$, $q_1=q_2=q_3=0$.
One might show by an argument of "equivariant deformation to the
normal cone", similar to the one used in \cite{L-S}, that $\tilde{\ob}_2^{-1}(0)$ is isomorphic to $\tilde{\ob}^{-1}(0)$, by an isomorphism, equivariant under the action of $\PP\Aut(\EEE,\nabla)\simeq\PGL(2,\CC)$, so that one can use $\tilde{\ob}_2^{-1}(0)$ instead of $\tilde{\ob}^{-1}(0)$ in order to determine the local structure of the moduli space. But, we will use another approach, proving directly that all the higher obstructions vanish in
our case: $\tilde{\ob}_3=\tilde{\ob}_4=\dots=0$ (see Sect. \ref{sect. 5.3}). 

Now we will determine the GIT quotient $Q:=\tilde{\ob}_2^{-1}(0)//\PGL(2,\CC)$.
We will proceed to a change of notation for the coordinates in $\CC^6$: 
$(x,x_{12},x_{21})=(s_0,s_1,s_2)$ and $(y,y_{12},y_{21})=(t_0,t_1,t_2).$
Then the three quadrics $q_1=q_2=q_3=0$ become $s_it_j=s_jt_i$. These equations express the proportionality $s=(s_0,s_1,s_2)\thicksim t=(t_0,t_1,t_2).$
The space of solutions is of dimension $4$; it can be expressed as the affine cone over the image of Segre $\PP^2\times\PP^1\hookrightarrow\PP^5:$\\
\begin{eqnarray*}
 \textrm{Cone}\Big(\PP^2\times\PP^1\Big) & \  \ra & \tilde{\ob}^{-1}_2\{0\}  \\ \nonumber
((\xi_0,\xi_1,\xi_2), (\lambda_0,\lambda_1) &  \mapsto\ & (s,t)=(\lambda_0\xi,\lambda_1\xi).\end{eqnarray*}
To determine the quotient of this cone by $\PGL(2,\CC)$, remark that the latter acts by simultaneous conjugation on the two components:
$$\PGL(2,\CC)\ni g:(A,a)\mapsto (gAg^{-1},gag^{-1}).$$
Consider the map
\begin{eqnarray*}
\phi:\tilde{\ob}^{-1}_2(0) & \  \ra & \CC^2 \\ \nonumber
(A,a) &  \mapsto\ & (z_1,z_2)=(\Tr(A^2),\Tr(a^2))
.\end{eqnarray*} 
It is given by $PGL(2,\CC)$-invariant functions, hence descends to a morphism $\psi:Q=\tilde{\ob}^{-1}_2(0)//\PGL(2)\rar\CC^2$. 
For a fixed pair $(z_1,z_2)$ with $z_1z_2\neq 0$, there are two orbits of pairs of commuting traceless matrices $(A,a)$ such that
$(\Tr(A^2),\Tr(a^2))=(z_1,z_2)$: these are orbits of two pairs \begin{eqnarray*} 
(A,a)=\Big\{\left(\begin{array}{cc} \frac{\sqrt{z_1}}{\sqrt{2}} & 0 \\ 0 & -\frac{\sqrt{z_1}}{\sqrt{2}} \end{array}\right),\pm\left(\begin{array}{cc} \frac{\sqrt{z_2}}{\sqrt{2}} & 0 \\ 0 & -\frac{\sqrt{z_2}}{\sqrt{2}} \end{array}\right)\Big\}. \end{eqnarray*}
They are distinguished by the $\PGL(2)$-invariant function $z=\Tr(Aa)$, taking opposite signs on them. There is only one orbit
in the fiber of $\psi$ if $z_1z_2=0$. Thus the map  \begin{eqnarray*}
\tilde{\psi}:Q=\tilde{\ob}^{-1}_2(0)//PGL(2,\CC) & \  \ra & Q_0=\{(z,z_1,z_2)\vert z^2=z_1z_2\} \\ \nonumber
(A,a) &  \mapsto\ & (\Tr(Aa),\Tr(A^2),\Tr(a^2))
,\end{eqnarray*} is a regular morphism which is bijective.  
As $Q_0$ is a normal variety,  $\tilde{\psi}$ is an isomorphism.
We have proved the following result:
\begin{theorem}
Let $X$ be an elliptic curve, $\EEE=\OOO_X\oplus\OOO_X$, $\nabla=d$ the de Rham connection on $\EEE$. Then the germ of 
$\tilde{\CCC}_X(2,0;0)$ at $(\EEE,\nabla)$ is analytically isomorphic to the germ $(\CC^2\times Q_0,0)$, where $Q_0$ is the
quadratic cone in $\CC^3$ with equation $z^2=z_1z_2$.
\end{theorem}
We will end this subsection by determining the fiber of the natural map $\pi:\tilde{\CCC}_X(2,0;0)\rar M_X(2,0)$
over $[\EEE]$, the point representing the $S$-equivalence class of $\EEE=\OOO_X\oplus\OOO_X$. We represent the germ of 
$M_X(2,0)$ at $[\EEE]$ in the same manner as we did for $\tilde{\CCC}_X(2,0;0)$: this is the quotient of $sl_2(\CC)=\End(\CC^2)_0$
(the $a$-component of the pairs $(A,a)$) by the action of $\PGL(2,\CC)$.
$$\PGL(2,\CC)\ni g:a\mapsto gag^{-1}.$$
The quotient is just the affine space $\CC^2$ with coordinates $$\frac{1}{2}\Tr(a)=y_0, \Tr(a^2)=z_2.$$
Hence the germ of $\pi$ at $(\EEE,\nabla)$ is given by 
\begin{eqnarray*}
\pi:\CC^2\times Q_0 & \  \ra & \CC^2 \\ \nonumber
\Big((x_0,y_0),(z,z_1,z_2)\Big) &  \mapsto\ & (y_0,z_2)
,\end{eqnarray*} and the fiber over $[\EEE]=(0,0)\in\CC^2$ is 
\begin{eqnarray}\label{eQa}\CC^2=\{z=z_2=y_0=0\}\end{eqnarray} taken with multiplicity $2$, for the equation $z^2=z_1z_2$ reduces to $z^2$ modulo $z_2$.
The fact that it occurs with multiplicity $2$ implies that $\pi$ is not a locally trivial fibration near $[\EEE]$, though it is still equidimensional.

We remark also that the formulas (\ref{eQa}) describe the fiber in $\pi^{-1}([\EEE])$ only locally near $(\EEE,\nabla)$, but it is obvious that this fiber is globally identified with $H^0(X,\mathcal{E}nd(\EEE)\otimes\Omega^{1}_X)_0/\PP(\Aut(\EEE))\simeq\End(\CC^2)_0/\PGL(2,\CC)\simeq\CC^2.$
We conclude:
\begin{proposition}
In the situation of Theorem \ref{struc}, let $\pi:\tilde{\CCC}_X(2,0;0)\rar M_X(2,0)$ be the natural map of forgetting the connection component
of each pair $(\EEE,\nabla)$. Then the fiber of $\pi$ over the most degenerate point $\EEE=\OOO_X\oplus\OOO_X$ of $M_X(2,0)$ is
isomorphic to the affine plane $\CC^2$ taken with multiplicity~$2$.
\end{proposition}

\subsection{Vanishing of higher obstructions}\label{sect. 5.3}
In Sect. \ref{sect. 5.2}, we stated that when determining the local structure of the moduli space $\tilde{\CCC}_X(2,0;0)$ at
$(\EEE,\nabla)=(\OOO_X\oplus\OOO_X,d)$, one can replace the obstruction map $\ob:\HH^1(X,\CCC^{\fatdot})\rar\HH^2(X,\CCC^{\fatdot})$
by its initial term $\ob_2$, which is quadratic on $\HH^1(X,\CCC^{\fatdot})$. We noted that this might be proved by an argument of
\cite{L-S}, but here we provide a simpler proof of a stronger assertion: the higher obstructions $\ob_k$ vanish for all $k\geq 3$.

As in Sect. \ref{sect. 5.2}, we will denote by $x_0,x,x_{12},x_{21}$ coordinates on $H^0(X,\mathcal{E}nd(\EEE)\otimes\Omega_X^1)$ and
by $y_0,y,y_{12},y_{21}$ coordinates on $H^1(X,\mathcal{E}nd(\EEE))$. Let us fix a flat complex coordinate $z$ on $X$ and denote by $P$
the point $z=0$. To compute the cohomology of (complexes of) coherent sheaves on $X$, we will use the two-element Stein open covering
$X=U_0\cup U_1$, where $U_0=\{\left|z\right|<\epsilon\}$ is a small disc centered at $P$, and $U_1=X\setminus\{P\}$. We denote a basis
of $H^1(X,\OOO_X)$ by $\delta$, as in Sect. \ref{sect. 5.2}, and will represent it by the $\check{C}$ech cocycle $\delta_{\alpha\beta}=\frac{1}{z}\in
\Gamma(U_{\alpha}\cap U_{\beta},\OOO_X).$
Here and in the sequel, $\alpha=0,\beta=1$. We will also specify a section $\sigma:H^1(X,\OOO_X)\rar\HH^1(X,\Omega_X^{\fatdot})$ on the level of
$\check{C}$ech cocycles. A cocycle defining $\sigma(\delta)$ is given by $\phi_{\gamma}\in\Gamma(U_{\gamma},\Omega_X^1)$ $(\gamma=0,1)$ such that
$$\phi_{\beta}-\phi_{\alpha}=\ud(\delta_{\alpha\beta})=-\frac{\ud z}{z^2}.$$
\begin{lemma}\label{phi}
For any $k\in\ZZ\setminus\{-1\}$, there exists a $0$-cochain $\phi_k=(\phi_{\gamma,k})\in\check{C}^0(\mathfrak{U},\Omega_X^1)$
such that $$\phi_{\beta,k}-\phi_{\alpha,k}=\frac{\ud z}{z^k}$$ on $U_{\alpha}\cap U_{\beta},$ where $\mathfrak{U}=(U_0,U_1)$ denotes
the two-element covering.
\end{lemma}
\begin{proof}
The Weierstrass $\wp$-function has the following Laurent expansion at $0$:
$$\wp(z)=\frac{1}{z^2}+c_2z^2+c_4z^4+\dots.$$
It is regular on $U_1$. We can set 
$$\phi_{\beta,k}=\frac{(-1)^k}{(k-1)!}\wp^{(k-2)}(z)\ud z,\ \ \phi_{\alpha,k}=-\frac{\ud z}{z^k}+\phi_{\beta,k}\ \ \ \  (k\geq 2).$$
\end{proof}
The choice of $(\phi_{\gamma,k})$ is not unique. Let us fix such choices for all $k\geq 2$ once and for all;
set $\phi_{\gamma}=-\phi_{\gamma,2}$.
Now we are ready to fix a choice of coordinates on $\HH^1(X,\CCC^{\fatdot})$: these are $8$ linear forms
$x_0,x,x_{12},x_{21},y_0,y,y_{12},y_{21}$ assembled into two matrices 
\begin{eqnarray*}
T=\left(\begin{array}{cc} x_0+x & x_{12} \\ x_{21} & x_0-x \end{array}\right),
\qquad Y=\left( \begin{array}{cc} y_0+y & y_{12} \\ y_{21} & y_0-y \end{array}\right),\end{eqnarray*}
The pair $(A,a)=(T\ud z,Y\delta)$ corresponds to the cohomology class, represented by the cocycle 
\arraycolsep=0.5ex
\begin{eqnarray*}
\big((T-\phi_{\alpha,2}Y), (T-\phi_{\beta,2}Y), \frac{1}{z}Y\big)\in\Gamma(U_{\alpha},\mathcal{E}nd(\EEE)\otimes\Omega_X^1)\oplus\Gamma(U_{\beta},\mathcal{E}nd(\EEE)\otimes\Omega_X^1)\\ \nonumber
\oplus\Gamma(U_{\alpha}\cap U_{\beta},\mathcal{E}nd(\EEE)).\end{eqnarray*}
An order-$k$ deformation of the trivial bundle $\EEE=\OOO\oplus\OOO$ will be represented by a cocycle 
$G_{\alpha\beta}^{(k)}=G_{\alpha\beta,0}+G_{\alpha\beta,1}+\dots+G_{\alpha\beta,k}$, where $G_{\alpha\beta,0}=\id_2$ and $G_{\alpha\beta,k}\in\Gamma(U_{\alpha}\cap U_{\beta},\mathcal{E}nd(\EEE))\otimes
k[y_0,y,y_{12},y_{21}]$ is homogeneous of degree $k$ in  $y_0,y,y_{12},y_{21}$.
The forms $G_{\alpha\beta,k}$ are not subject to any constraints (cocycle condition), because our covering has no triple intersections.
Similarly, an order-$k$ deformation of $\nabla$ will be given by connection matrices $A_{\gamma}^{(k)}$,
$$A_{\gamma}^{(k)}=A_{\gamma,0}+A_{\gamma,1}+\dots+A_{\gamma,k} \ (\gamma=0,1),$$
where $A_{\gamma,0}=0$, and $A_{\gamma,k}\in\Gamma(U_{\gamma},\mathcal{E}nd(\EEE)\otimes\Omega_X^1)\otimes k[x_0,x,x_{12},x_{21},y_0,y,y_{12},y_{21}]$
is homogeneous of degree $k$.
\begin{theorem}
There exist sequences $(G_{\alpha\beta,k})_{k\geq 0}, (A_{0,k})_{k\geq 0},(A_{1,k})_{k\geq 0}$, where 
$G_{\alpha\beta,k}\in\Gamma(U_{\alpha}\cap U_{\beta},\mathcal{E}nd(\EEE))\otimes k[x_0,x,x_{12},x_{21}],
A_{\gamma,k}\in\Gamma(U_{\gamma},\mathcal{E}nd(\EEE)\otimes\Omega_X^1)\otimes k[x_0,x,x_{12},x_{21},y_0,y,y_{12},y_{21}],$
are homogeneous of degree $k$ ($\alpha=0,\beta=1,\gamma\in\{0,1\}),$ with the following properties\\
$(i)$ $G_{\alpha\beta,0}=\id_2,\ G_{\alpha\beta,1}=\frac{1}{z}Y.$\\
$(ii)$ $A_{\gamma,0}=0,\ A_{\gamma,1}=T-\phi_{\gamma,2}Y$ \ ($\gamma=0,1)$.\\
$(iii)$ For any $k\geq 1$, the polynomials $G_{\alpha\beta}^{(k)}=\sum_{i=0}^{k}G_{\alpha\beta,i}$ and $A_{\gamma}^{(k)}=\sum_{i=0}^{k}A_{\gamma,i}$
satisfy the following congruence:
$$\ud G_{\alpha\beta}^{(k)}\equiv G_{\alpha\beta}^{(k)}A_{\beta}^{(k)}-A_{\alpha}^{(k)}G_{\alpha\beta}^{(k)} \mod I_k,$$
where $I_k=(q_1,q_2,q_3)+\mathfrak{m}^{k+1},$ $q_1=x_{12}y_{21}-x_{21}y_{12}, q_2=2xy_{12}-2x_{12}y, q_3=2x_{21}y-2y_{21}x$,
$\mathfrak{m}=(x_0,x,x_{12},x_{21},y_0,y,y_{12},y_{21}).$
\end{theorem}
In other words, there are only quadratic obstructions to extending $(\EEE,\nabla)$ to any order $k\geq 2$.
\begin{proof}
As \begin{eqnarray}[T,Y]=\left( \begin{array}{cc} q_1 & q_2 \\ q_3 & -q_1\end{array}\right),\end{eqnarray}
we can handle $T,Y$ as commuting matrices. A solution to $(i)-(iii)$ can be represented by the following power series:
$$G_{\alpha\beta}=\sum_{k=0}^{\infty} G_{\alpha\beta,k}=e^{\frac{1}{z}Y}, A_{\gamma}=\sum_{k=0}^{\infty} A_{\gamma,k}=T-\phi_{\gamma,2}Y$$
(so that $A_{\gamma,k}=0$ for all $k\neq 1$.) Then we have
\arraycolsep=0.5ex
\begin{eqnarray*}
\ud G_{\alpha\beta}&=&-\frac{1}{z^2}Ye^{\frac{1}{z}Y}\ud z=(\phi_{\alpha,2}-\phi_{\beta,2})Ye^{\frac{1}{z}Y}\\
&=&e^{\frac{1}{z}Y}(T-\phi_{\beta,2}Y)-(T-\phi_{\alpha,2}Y)e^{\frac{1}{z}Y}\\
&=&G_{\alpha\beta}A_{\beta}-A_{\alpha}G_{\alpha\beta},\end{eqnarray*} (the computation is done modulo $(q_1,q_2,q_3)$, so $T$ and $Y$ commute.)
\end{proof}
\begin{remark}
One might also set $G_{\alpha\beta,1}=\frac{1}{z}Y,$ $G_{\alpha\beta,k}=0\ \forall k\geq 2$, and construct $A_{\alpha,k}$ by
induction on $k$, using the $\phi_{\gamma,k}$ from Lemma \ref{phi}.
\end{remark}
\subsection{The local structure of $\tilde{\CCC}_X(2,0;0)$ for $g=1$: mildly degenerate case}
In the setting of Sect. \ref{sect. 5.2}. and \ref{sect. 5.3}, let now $\EEE=\OOO_X\oplus\OOO_X$ and $\nabla=d+A$, with
$A\in H^0(X,\mathcal{E}nd(\EEE)\otimes\Omega_X^1)$ generic. Changing a basis for $\EEE$, we can assume $A$ diagonal:
\begin{eqnarray*}A=\left( \begin{array}{cc} \lambda_1 & 0 \\ 0 & \lambda_2 \end{array}\right)\ud z,\end{eqnarray*}
$\lambda_1\lambda_2\neq 0,\lambda_1\neq\lambda_2$.
Then the map $d_1$ on both $H^i(X,\mathcal{E}nd(\EEE))$ in the exact sequence (\ref{lgexs}) is given by $B\mapsto [A,B]=
A\circ B-B\circ A$, where $\circ$ denotes the Yoneda composition, and it has a two-dimensional kernel and $2$-dimensional
image. Thus we have the exact triple
\begin{eqnarray*}
0\rar M_2(\CC)\ud z/\left(\begin{array}{cc} 0 & \CC\ud z \\ \CC\ud z & 0\end{array}\right)\rar\HH^1(X,\CCC^{\fatdot})\rar\left(\begin{array}{cc} \CC\delta  & 0 \\ 0 & \CC\delta\end{array}\right)\rar 0\end{eqnarray*}
and \begin{displaymath}\HH^0(X,\CCC^{\fatdot})=\left( \begin{array}{cc} \CC\ud z & 0 \\ 0 & \CC\ud z \end{array}\right)\end{displaymath}
\begin{displaymath}\HH^2(X,\CCC^{\fatdot})=M_2(\CC)\delta\ud z/\left( \begin{array}{cc} 0 & \CC\delta \ud z  \\ \CC\delta \ud z & 0  \end{array}\right).\end{displaymath}
The obstruction map always takes values in the traceless part $\HH^2(X,\CCC^{\fatdot}_0)$, where
\begin{equation*} \CCC^{\bullet}_0=[\mathcal{E}nd_{0}(\EEE)\xymatrix@1{\ar[r]^{\nabla}&}\mathcal{E}nd_{0}(\EEE)\otimes_{\OOO_X}\Omega^1_X],\end{equation*}
which is $1$-dimensional in the case under consideration. Thus we might expect that the base of the versal deformation
is a hypersurface in the $4$-dimensional space $\HH^1(X,\CCC^{\fatdot})$. But this is impossible, for we can construct by hand
a $4$-dimensional family of connections, which is a deformation of $(\EEE,\nabla)$ and all of whose members are pairwise non-isomorphic.
Indeed, $\OOO_X\oplus\OOO_X$ deforms to $\LLL_1\oplus\LLL_2$ with $\LLL_i\in\Pic^0(X)$; each of the $L_i$ has a connection $d_i$, and we
get two more parameters in adding arbitrary multiples of $\ud z$ to $d_i$. The family obtained 
$(\LLL_1\oplus\LLL_2,(d_1+\alpha_1\ud z)\oplus(d_2+\alpha_2)\ud z)$ is the wanted $4$-parameterized deformation of $(\OOO_X\oplus\OOO_X,\nabla).$
We conclude that the obstruction map is zero, and moreover, as the members of our family are polystable connections, the base of the
family injects into the moduli space $\tilde{\CCC}_X(2,0;0)$. This implies that though the automorphism group $\PP(\Aut(\EEE,\nabla))\subset\PGL(2,\CC)$
is non-trivial here, it acts trivially on $\HH^1(X,\CCC^{\fatdot})$, so that the neighborhood of $0$ in $\HH^1(X,\CCC^{\fatdot})$ is a local chart
for $\tilde{\CCC}_X(2,0;0)$ at $(\EEE,\nabla)$ (in the classical or \'etale topology). We summarize this in the following statement.
\begin{theorem}
Let $X$ be an elliptic curve, $\EEE=\OOO_X\oplus\OOO_X$, $\nabla=d+A$ with $A\in H^0(X,\mathcal{E}nd(\EEE)\otimes\Omega_X^1)$ generic. 
Then $\dim \HH^1(X,\CCC^{\fatdot})=4$, and the neighborhood of $0$ in $\HH^1(X,\CCC^{\fatdot})$ is a local chart
for $\tilde{\CCC}_X(2,0;0)$ at $(\EEE,\nabla)$, so that $\tilde{\CCC}_X(2,0;0)$ is smooth of dimension $4$ at $(\EEE,\nabla)$.
In particular, this means that the obstruction map $\ob:\HH^1(X,\CCC^{\fatdot})\rar\HH^2(X,\CCC^{\fatdot})$ vanishes and that $
\PP\Aut(\EEE,\nabla)$ acts trivially on $\HH^1(X,\CCC^{\fatdot})$.
\end{theorem}

We remark that the map $\pi:\tilde{\CCC}_X(2,0;0)\rar M_X(2,0)$ is not locally trivial near $(\EEE,\nabla)$, for $(\EEE,\nabla)$
belongs to the same fiber $\pi^{-1}([\OOO_X\oplus\OOO_X])$ which we showed to be a multiple one. The failure of local triviality
is due to the fact that $\Aut(\EEE)$ is bigger than $\Aut(\EEE,\nabla)$. 
\subsection{Direct images from bielliptic curves}
Let $E$ be an elliptic curve, $f:C\rar E$ a bielliptic cover as in Section 2 of \cite{Machu-1}, $\LLL$ a line bundle on $C$
and $(\EEE,\nabla)=(f_*(\LLL),f_*(\nabla_\LLL))$ a direct image connection on $E$. We will use the notation from Section 4 of \cite{Machu-1}.
Let $\CCC_X$ denote the moduli space $\CCC_X(2,-1;p_+ +p_-)$, where $p_{\pm}$ are branch points of $f$. We will investigate $\CCC_X$ in
the neighborhood of $(\EEE,\nabla)$. We will first assume that $\EEE$ is stable.
\begin{lemma} If $\EEE$ is stable, then
$\dim H^i(E,\mathcal{E}nd(\EEE))=1 \ (i=0,1)$, and $\dim H^0(E,\mathcal{E}nd(\EEE)\otimes\Omega_X^1(D))=8$, $\dim H^1(E,\mathcal{E}nd(\EEE)\otimes\Omega_X^1(D))=0$, where $D=p_+ +p_-$.
\end{lemma}
\begin{proof}
$\EEE$ can be obtained as an extension 
\begin{equation}\label{exte}
0\rar\OOO_E(-p)\rar\EEE\rar\OOO_E\rar 0
,\end{equation}
where $p\in E$ is such that $\det(\EEE)\simeq\OOO_E(-p)$. The Bockstein homomorphism 
$\partial:H^0(E,\OOO_E)\rar H^1(E,\OOO_E(-p))$ is the product with the extension class of (\ref{exte}), hence an isomorphism, otherwise the extension would be trivial
and $\EEE$ would be unstable. Twisting (\ref{exte}) by $\OOO_E(D)$ and $\OOO_E(D+p)$, we find $\dim H^0(E,\EEE(D))=3$,
$\dim H^0(E,\EEE(D+p))=5$, and  $\dim H^1(E,\EEE(D))=\dim H^0(E,\EEE(D+p))=0$.
Finally, tensoring (\ref{exte}) by $\EEE^* (D)\simeq\EEE(D+p)$, we obtain the exact triple
$$0\rar\EEE(D)\rar\mathcal{E}nd(\EEE)\otimes\OOO_E(D)\rar\EEE(D+p)\rar 0,$$ which implies the result since $\Omega^1_E\simeq\OOO_E$.
\end{proof}
Thus the exact sequence (\ref{lgexs}) implies that  $\HH^2(E,\CCC^{\fatdot})=0$, so that the infinitesimal deformations
are unobstructed. Further, 
$$\HH^0(E,\CCC^{\fatdot})\simeq H^0(E, \mathcal{E}nd(\EEE))\simeq\CC,$$ (by stability),
and the triple
\begin{equation}\label{exect}
0\rar\ H^0(E,\mathcal{E}nd(\EEE)\otimes\Omega_X^1(D))\rar\HH^1(E,\CCC^{\fatdot})\rar H^1(E,\mathcal{E}nd(\EEE))\rar 0
\end{equation} is exact.
Hence $\dim\HH^1(E,\CCC^{\fatdot})=9$, the base of the versal deformation is smooth, and $\Aut(\EEE,\nabla)$ is trivial.
Similarly, deformations of $\EEE$ are unobstructed and $\Aut\EEE$ is trivial.
We deduce:

\begin{proposition}
In the above notation, the neighborhood of zero in $\HH^1(E,\CCC^{\fatdot})$ (resp. $H^1(X,\mathcal{E}nd(\EEE))$) is a local chart
for $\CCC^0_E=\CCC^0_E(2,-1; D)$ (resp. $M_E(2,-1)$), and the natural projection $\pi:\CCC^0_E\rar M_E(2,-1)$,
$(\EEE,\nabla)\mapsto [\EEE]$, is represented in these charts by the natural epimorphism $\HH^1(E,\CCC^{\fatdot})\rar H^1(E,\mathcal{E}nd(\EEE))$
coming from the exact triple (\ref{exect}). We have $M_E(2,-1)\simeq E$, and $\pi$ is an affine bundle over $E$ with fiber $\CC^8$, locally trivial
in the Zariski topology.
\end{proposition}
\begin{proof}
The existence of universal families of pairs $(\EEE,\nabla)$ over Luna slices locally in the \'etale topology
implies that $\pi$ is an affine bundle, locally trivial in the \'etale topology. But according to Grothendieck \cite{Gro-1}, over a curve,
any affine bundle, locally trivial in the \'etale topology is also locally trivial in the Zariski topology; this follows
from the vanishing of the Brauer group of any smooth projective curve.
\end{proof}

Now, we turn to the case where $\EEE$ is unstable. Then $\pi:\CCC_E\dasharrow M_E(2,-1)$  exists as a rational map in a neighborhood
of $(\EEE,\nabla)$, and it is undefined at $(\EEE,\nabla)$. Then there exists $\LLL\in\Pic^0(E)$ (see Proposition 13 of \cite{Machu-1}) such that
$\EEE=\OOO_E(-\infty)\oplus\LLL$. We have $\mathcal{E}nd(\EEE)\simeq\OOO^{\oplus 2}\oplus\LLL^*(-\infty)\oplus\LLL(\infty)$,
hence $\dim H^i(\mathcal{E}nd(\EEE))=3\ (i=0,1)$, $\dim H^0(E,\mathcal{E}nd(\EEE)\otimes\Omega_X^1(D))=8$ and
$\dim H^1(E,\mathcal{E}nd(\EEE)\otimes\Omega_X^1(D))=0$. Hence $\HH^2(E,\CCC_E^{\fatdot})=0$ and the deformations of 
$(\EEE,\nabla)$ are unobstructed. The description of $\HH^1(E,\CCC_E^{\fatdot})$  depends on the rank of the map
$d_1: H^0(\mathcal{E}nd(\EEE))\rar H^0(E,\mathcal{E}nd(\EEE)\otimes\Omega_X^1(D))$ in the long exact sequence (\ref{lgexs}).
For instance, if $\nabla$ is the direct image of the regular connection $\ud+\omega$ on the trivial bundle $\OOO_C$,
we see from formulas 8-10 of Section 4 of \cite{Machu-1} that \begin{displaymath}d_1:B=\left( \begin{array}{cc} b_{11} & b_{12} \\ 0 & b_{22} \end{array}\right)\mapsto [A,B] \end{displaymath} is of rank $1$ if $\lambda_1=0$ and of rank $2$ otherwise. But if $\lambda_1=0$, then $(\EEE,\nabla)$ is the direct sum 
of rank-$1$ connections and hence is unstable. Thus it does not represent a point of $\CCC_E$. Similarly, using formulas
from Proposition 13 of \cite{Machu-1}, we can show that $\rk d_1=2$ in all the remaining cases in which $(\EEE,\nabla)$ is stable, but $\EEE$ is not.
We conclude that there is the exact triple
\arraycolsep=0.5ex
\begin{eqnarray*} 0&\rar& H^0(E,\mathcal{E}nd(\EEE)\otimes\Omega_X^1(D))/d_1(H^0(E,\mathcal{E}nd(\EEE))\simeq\CC^6\rar\HH^1(E,\CCC_E^{\fatdot})\simeq\CC^9\\
&\rar& H^0(E,\mathcal{E}nd(\EEE))\simeq\CC^3\rar 0.\end{eqnarray*}
Thus $\CCC_E$ is smooth of dimension $9$, as in the case when $\EEE$ is stable.
Note that in the base of the versal deformation of $(\EEE,\nabla)$, there is a hypersurface parameterizing
the deformations of $(\EEE,\nabla)$ with $\EEE$ unstable. It is defined by the vanishing of the component 
corresponding to the factor $H^1(E,\LLL^{*}(-\infty))\subset H^1(E,\mathcal{E}nd(\EEE))$.
The complement to this $8$-dimensional hypersurface corresponds to stable $\EEE$'s, so that $\CCC_E$ can be viewed
as a partial compactification of $\CCC^0_E$ at infinity (at the limit $"\nabla\rar\infty"$).
Moreover as $\det:M_E(2,-1)\rar E$ is an isomorphism, and $\det$ is well-defined on $\CCC_E$, so that $\CCC_E\setminus\CCC_E^0$ is 
a removable singularity of $\pi$,  locally near $(\EEE,\nabla)$.
We conclude:
\begin{theorem}
Let $(\EEE,\nabla)$ be a stable direct image connection, as in Section 4 of \cite{Machu-1}, with $\EEE$ unstable.
Then the rational map $\pi:\CCC_E\dasharrow M_E(2,-1)$ is regular at $(\EEE,\nabla)$ and is a smooth partial compactification of 
the affine bundle $\CCC_E^0\rar M_E(2,-1)$ with fiber $\CC^8$.

The infinitesimal deformations of $(\EEE,\nabla)$ are unobstructed, and $\HH^1(E,\CCC^{\fatdot})\simeq\CC^9$ is
a local chart of $\CCC_E$ near $(\EEE,\nabla)$.
\end{theorem} 

We remark that the local triviality of $\pi$ is not global. For example, $\CCC_E$ contains a pair $(\EEE,\nabla)$
with $\EEE=(\OOO_E(-\infty),\LLL)$ as above, and \begin{displaymath}A=\left( \begin{array}{cc} 0 & \ud z \\ 0 & 0 \end{array}\right)
.\end{displaymath}
This pair is stable, since the unique destabilizing subbundle $\LLL$ in $\EEE$ is not $\nabla$-invariant, but $\rk d_1=1$ and
$\dim\HH^1(\CCC^{\fatdot})=10>\dim \CCC_E$.
\subsection*{Acknowledgements}
I am greatly indebted to my former research advisor D.~Markushevich for his encouragement and help.

\renewcommand\refname{References}

\end{document}